\documentclass{lms}
\usepackage{amsmath, amssymb, multirow}

\newtheorem{theorem}{Theorem}[section] 
\newtheorem{lemma}[theorem]{Lemma}     

\newtheorem{proposition}[theorem]{Proposition}
\newnumbered{assertion}{Assertion}    
\newnumbered{conjecture}{Conjecture}  
\newnumbered{definition}{Definition}
\newnumbered{hypothesis}{Hypothesis}
\newnumbered{remark}{Remark}
\newnumbered{algorithm}{Algorithm}
\newnumbered{note}{Note}
\newnumbered{observation}{Observation}
\newnumbered{problem}{Problem}
\newnumbered{question}{Question}
\newnumbered{example}{Example}
\newunnumbered{notation}{Notation} 

\newcommand     {\GL}[1][\Z]    {\sym{GL}(2,{#1})}
\newcommand     {\C}            {{\mathbb C}}

\newcommand     {\F}            {{\mathbb F}}
\newcommand     {\Q}            {{\mathbb Q}}
\newcommand     {\Sone}         {{\mathbb S}^1}
\newcommand     {\SL}[1][\Z]    {\sym{SL}(2,{#1})}
\newcommand     {\Sp}[1][\Z]    {\sym{Sp}(2,{#1})}
\newcommand     {\Z}            {{\mathbb Z}}
\newcommand     {\N}            {\ensuremath{\mathbb{N}}}

\newcommand     {\qf}[3]        {[#1,#2,#3]}
\newcommand     {\smat}[1]      {\left(\begin{smallmatrix}#1\end{smallmatrix}\right)}
\newcommand     {\tr}           {\sym{tr}}
\newcommand     {\sym}[1]       {\operatorname{#1}}
\newcommand     {\PH}           {{\mathbb H}}
\newcommand     {\disc}         {\textrm{disc}}
\newcommand     {\kro}[2]       {{\left(\frac{#1}{#2}\right)}}
\newcommand     {\isdiv}        {\ensuremath{\mathop{\mid}}}
\newcommand     {\slashdiv}     {\ensuremath{\mathop{/}}}
\renewcommand     {\det}          {\ensuremath{\mathrm{det}}}
\newcommand     {\rsym}         {\ensuremath{\mathrm{sym}}}
\newcommand     {\set}[1]       {{\def\st{\;:\;}\left\{#1\right\}}}
\newcommand     {\Mat}[1]       {\ensuremath{\mathrm{M}_{#1}}}

\title[Computation of Siegel modular forms]{%
  Explicit computations of Siegel modular forms of degree two.
}

\author[Raum, Ryan, Skoruppa, Tornar\'ia]{%
  Martin Raum, Nathan C. Ryan, Nils-Peter Skoruppa, and Gonzalo Tornar\'ia
  }

  \extraline{%
    This project was supported by the National Science Foundation
    under FRG Grant No. DMS-0757627}

\date{%
  \today
}

\classno{11F46}

\begin{document}

\maketitle

\begin{abstract}
Unlike classical modular forms, there is currently no general way
to implement the computation of Siegel modular forms of arbitrary
weight, level and character, even in degree two.  There is however, a
way to do it in a unified way.  After providing a survey of known
computations we describe the implementation of a class modeling
Siegel modular forms of degree two in Sage.  In particular, we
describe algorithms to compute a variety of rings of Siegel
modular forms, many of which are implemented in our class.
A wide variety of Siegel modular forms (e.g., both vector- and
scalar-valued) can be modeled via this class and we unify these via a
construct we call a formal Siegel modular form.  We define this notion
and discuss it in detail.

\end{abstract}

\section{Introduction}

Computing modular forms is of broad interest.  For example, the
explicit computation of elliptic modular forms has led to a number of
conjectures that in many cases have become theorems.  Similar, but
neither as comprehensive nor as systematic, computations have been
carried out with Siegel modular forms as well.  A consideration
of Siegel modular forms and their Fourier expansions aiming at an
implementation should suit two purposes.  Firstly, one has to obtain
the generators of a space or ring of modular forms as easily and
quickly as possible.  Secondly, one wants to construct more
complicated modular forms and recognize them as a linear or algebraic
combination of the generators.

Computing elliptic modular forms of level 1 is rather straightforward.
It is well-known that the ring of such forms is a polynomial ring
generated by the Eisenstein series $E_4$ and $E_6$.  Alternatively,
the weight $k$ space is spanned by $E_k$ and all the products $E_{i}
E_{k-i}$ for $i$ running from $4$ to $k-2$ as can be proved by means
of methods found in \cite{kohnenzagierperiods}.  Computations in level
bigger than one are typically done via the theory of modular
symbols.  For instance, this is how both MAGMA \cite{MAGMA} and Sage
\cite{Sage} compute spaces of elliptic modular forms.  

An elliptic modular form can be thought of as a Siegel modular form of
degree 1.  The theory of the computation of Siegel modular forms of
even degree 2 is not as general as that for degree 1.   There is
currently no theory of modular symbols.  What there is, though, is an
explicit description of a number of particular rings of Siegel modular
forms.  More generally, by a theorem of Igusa, we know that every ring
associated to any finite index subgroup of the full Siegel modular
group is generated by polynomials in theta constants
(cf. \cite{Igusa4,Igusa5}).  This theorem has been exploited by
Igusa \cite{Igusa2,Igusa3} to explicitly list the
5 generators of the ring of Siegel modular forms of degree 2.  A more
recent application of this theorem in a number of papers
\cite{Ibukiyama,Ibukiyama3,Ibukiyama2} by Ibukiyama and his collaborators have
characterized rings of Siegel modular forms for levels up to 4 as
polynomial rings in explicitly described generators.  These generators
are typically described in terms of theta constants.

Another method to compute spaces of Siegel modular forms (and thereby,
possibly rings of modular forms) is Poor and Yuen's restriction
technique (see for example, \cite{PoorYuen}).  In their papers on this
topic they are typically trying to compute forms in level 1 and degree
larger than 2, but in principle their technique should work in degree
2.  The potential obstacle when computing for level bigger than one is that there are
several cusps and one needs to know the expansion at each of them.
The method proceeds roughly as follows:  they define a linear map from
spaces of degree $n$ Siegel modular forms down to degree 1.  They
search for linear relations among the Fourier coefficients of degree 1
forms and pull them back to linear relations among the Fourier
coefficients of degree $n$ forms.  This allows them to compute the
dimensions of particular spaces of degree $n$ Siegel modular forms and
to write down bases for those spaces.  To determine these bases, they
compute theta series, Ikeda lifts, Eisenstein series, and theta
constants (see Section~\ref{sec:algorithms}).

The above two families of results are both for scalar-valued Siegel
modular forms.  In some cases (e.g., \cite{Satoh}), one can explicitly
compute spaces of vector-valued Siegel modular forms.  An example of
such a space is given in terms of the 4 Igusa generators of even
weight and an explicitly described differential operator (see
Section~\ref{ssec:satoh-bracket}).  We limit our attention to spaces of
vector-valued Siegel modular forms because the ring of such forms is
not finitely generated.

In addition to providing a useful survey of what is known about the
computation of degree 2 Siegel modular forms (e.g., what rings can be
computed, what vector spaces, etc.), we describe an implementation
\cite{smf-package} of the types of Siegel modular forms often used to
describe generators.  We provide algorithms that are used in our
implementation in Sage \cite{Sage} and that are improvements
(sometimes minor) over the naive algorithms that follow immediately
from theorems in the literature. Another feature of our implementation
is that it allows for algebraic combinations of vector- and scalar-valued Siegel modular forms.  In fact, our class was designed in such
a general way that one could use the same basic class to implement a
wide variety of automorphic forms.

Often one is interested in computing Hecke eigenvalues associated to
modular forms as opposed to computing Fourier coefficients.  In our
work we are mainly interested in computing Fourier coefficients but
now mention recent work that allows for the computation of Hecke
eigenvalues.  It is possible to compute Hecke eigenvalues by a purely
cohomological approach.  Until now this has been particularly useful
in the case of vector-valued Siegel modular forms
\cite{BergstroemFaberVanderGeer1,BergstroemFaberVanderGeer2}.  The
trace of the Hecke action on the full cohomology can be determined by
counting curves modulo $p$.  These calculations are demanding in terms of
theory and raw computational power.

The paper is organized as follows.  In Section~\ref{sec:siegel} we
give the definition of Siegel modular forms and provide the definition
of formal Siegel modular forms: this is what our Sage package actually
implements.  In Section~\ref{sec:algorithms}, we do two things when
appropriate.  First, for a wide
variety of types of Siegel modular forms of degree 2, we give a survey
of known theorems related to their computation.  Second, we give a
description of the algorithm we use to compute Siegel modular forms of
that type and any theorems related to the algorithm.  Then, in
Section~\ref{sec:rings} we list the spaces and rings that are
described in the literature and that we have computed.  The final two
sections deal with the implementation, its details and the choices we
have made.  We also provide links to data for modular forms in those
rings for which we have computed data.  In
Section~\ref{sec:implementation} we describe in broad strokes the
class we designed, paying particular attention to how we multiply
Siegel modular forms of different types and to how we model the
precision of a given form.  In Section~\ref{sec:generalizations} we
briefly point out how other spaces of automorphic forms can be modeled
via a class similar to the one we describe.


\section{Siegel Modular Forms}\label{sec:siegel}

In this section, we define Siegel modular forms of degree 2.   If one were to implement this definition directly, one would have a rather
inefficient class as it would be essentially a dense power series ring
in three variables. 
What we end up implementing in Sage is a
\emph{formal} Siegel modular, a notion we define in this section.
Also in this section we show that formal Siegel modular forms do model
the Fourier expansions of Siegel modular forms rather well.  A formal
Siegel modular form is in some sense more general than a Siegel
modular form: the set of Fourier expansions of Siegel modular forms is
a subset of the Fourier expansions modeled by formal Siegel modular
forms.  We sometimes take advantage of this fact in the
implementations of our algorithms.  At the same time, however, it respects as much structure as
needed to store the data associated to the modular form in a compact way.

\subsection{Basic Definitions}

We recall the definition of a Siegel modular form of degree two.  We
will consider the symplectic group $\textrm{Sp}(2,\Q)$ of degree $2$
over $\Q$; i.e., the group of matrices $M \in \textrm{M}(4, \Q)$ such
that ${}^t M J M = J$.  Here $J = \left(\begin{smallmatrix} & I_2 \\
    -I_2 & \end{smallmatrix}\right)$ represents the standard
symplectic form.  The full Siegel modular group $\textrm{Sp}(2, \Z) =
\textrm{Sp}(2, \Q) \cap \textrm{M}(4, \Z)$ will be the central object
of study.  For any $\Gamma\subset \textrm{Sp}(2,\Q)$ commensurable to
$\textrm{Sp}(2, \Z)$, let $\psi:\Gamma\to\Sone$ be a character and
$k,j$ be nonnegative integers.  Let
\begin{equation*}
  \PH_2
  :=
  \{Z=X+iY\in \textrm{M}(2,\C)\;:\;{}^tZ=Z,Y>0\}
\end{equation*}
be the Siegel upper half space of degree 2.  Assign to
$\C[X,Y]_j$, the space of homogeneous polynomials of degree $j$,
the $\GL[\C]$-action
\begin{equation}
  \label{eq:action}
  (A,p)\mapsto A\cdot p := p\bigl((X,Y)A\bigr).
\end{equation}

Now we can define Siegel modular forms:

\begin{definition}
  \label{def:smfs}
  A Siegel modular form of degree 2, weight $(k,j)$, and character
  $\psi$ on $\Gamma$ is a complex analytic function $F:\PH_2\to
  \C[X,Y]_j$ such that
  \begin{equation*}
    F(gZ)
    :=
    F\bigl((AZ+B)(CZ+D)^{-1}\bigr) = \psi(g)\,\det(CZ+D)^k\,(CZ+D)\cdot F(Z)
  \end{equation*}
  for all
  $g=\left(\begin{smallmatrix}A&B\\C&D\end{smallmatrix}\right)\in\Gamma$.
\end{definition}
The space of all such functions is denoted $M_{k,j}(\Gamma,\psi)$,
where we suppress $\psi$ if it is the trivial character and $j$ if it
is $0$. If $j$ is positive $F$ is called vector-valued,
otherwise it is called scalar-valued.
We use $M_*(\Gamma):=\bigoplus_{k} M_k(\Gamma)$ for the ring of (scalar-valued) Siegel modular
forms of degree 2 on $\Gamma$.

Common choices of $\Gamma$ and $\psi$ include $\Gamma = \Gamma_0(\ell)
:= \left(\begin{smallmatrix}\Z^{2\times 2} & \Z^{2\times 2}\\
    \ell\Z^{2\times 2} & \Z^{2\times 2}\end{smallmatrix}\right)\cap
\textrm{Sp}(2,\Z),$ and
$\psi\left(\begin{smallmatrix}A&B\\C&D\end{smallmatrix}\right):=\tilde{\psi}(\det
D)$ where $\tilde{\psi}$ is a Dirichlet character mod $\ell$.
A Siegel modular form on $\Gamma_0(\ell)$ and with such a character $\psi$
is commonly referred to as a Siegel modular form of level $\ell$ and
character $\tilde{\psi}$.


\subsection{Fourier Coefficients and Formal Siegel Modular Forms}
\label{ssec:formal}

Having defined analytic Siegel modular forms of degree 2, we now define a formal
Siegel modular form.  Every analytic Siegel modular form $F$ can be
identified by its Fourier expansion.  Because of a certain invariance
(see Proposition~\ref{prop:conclusion} for a precise statement) that
the coefficients of $F$ possess, the form $F$ can be identified by its
Fourier coefficients indexed on $\GL$-reduced quadratic forms.  So, we
define a formal Siegel modular form $C_F$ related to $F$ as a function
from integral binary quadratic forms to the representation space of
$F$.  

We denote by $Q$ the set of all integral binary quadratic forms that are
positive semidefinite:
\[
Q := \set{f = [a,b,c] \st a,b,c\in\Z,\; b^2-4ac\leq 0,\; a\geq 0}.
\]
For $f=[a,b,c]$ we use $M_f=\smat{a&b/2\\b/2&c}$.
The group $\GL$ acts on $Q$ by $M_{A\cdot f}=A\,M_f\,{}^t\!A$.

Let $F$ be a Siegel modular form on $\Gamma$.
It is not a serious restriction to assume, as we shall do for the
following, that
$P:=\left(\begin{smallmatrix} I_{2} & \Z^{2\times 2}\\
    0_{2}&I_{2}\end{smallmatrix}\right)\cap \textrm{Sp}(2,\Z) \subset
\Gamma,$ since this can always be achieved by passing from $\Gamma$ to
a $\textrm{Sp}(2,\Q)$-conjugate subgroup.  Moreover, we also assume that
$\psi$ is trivial on $P$.  The invariance of $F$
under the action of $P$ implies that $F$ is periodic and therefore has
a Fourier expansion of the form
\begin{align*}
  F(Z)
  &= \sum_{f=\qf abc\in Q} C_F(f)\,e\left(a\tau+bz+c\tau'\right) 
  \\
  &= \sum_{f=\qf abc\in Q} C_F(f)\,e\left(\tr (Z M_f)\right)
\end{align*}
Here $Z:=\left(\begin{smallmatrix} \tau &
    z\\z&\tau'\end{smallmatrix}\right)$ ($\tau,\tau'\in\PH_1$ and
$z\in\C$), $e(x)=e^{2\pi i x}$,
    and the trace of a matrix $A$ is denoted by $\tr A$.

Let $A^*:={}^t(A^{-1})$. Assume
\begin{equation}
  \label{eq:assumption}
  \left(\begin{smallmatrix} A^* & 0\\ 0 &
      A\end{smallmatrix}\right)\in\Gamma
  \text{ for some }A\in GL(2,\Z)
  .
\end{equation}
Then,
\begin{equation*}
    F(\smat{A^*&0\\0&A}Z)
  =
  F(A^*ZA^{-1})
  =
  \psi\left(\left(\begin{smallmatrix}A^* &
        0\\ 0 &A\end{smallmatrix}\right)\right)\det(A)^k\, A\cdot F(Z)
  .
\end{equation*}
In other words, equating the Fourier expansion of the left-hand side
to that of the right-hand side, we get
\begin{equation}
  \label{eq:coeffs}
  \sum_{f} C_F(f) \, e(\tr (A^* Z A^{-1}M_f))
  =
  \psi\left(\left(\begin{smallmatrix}A^* & 0\\ 0 &A\end{smallmatrix}\right)\right)
  \det(A)^k \, \sum_{f} A\cdot C_F(f) \, e(\tr( ZM_f)) .
\end{equation}
Using the fact that the trace is invariant under cyclic permutations
of the matrices being multiplied, we deduce
\begin{equation*}
  \tr(A^*Z A^{-1}M_f)=\tr(ZA^{-1}M_f A^*) = \tr(ZM_{A^{-1}\cdot f})
\end{equation*}
where the last equality follows from the definition of the $\GL$
action on $Q$.  Replacing $f$ by $A\cdot f$ in the left-hand side of
\eqref{eq:coeffs} yields the following:
\begin{proposition}
  \label{prop:conclusion}
  Let the notation be as above.  Then
  \begin{equation}
    \label{eq:conclusion}
    C_F(A\cdot f)
    =
    \psi\left(\left(\begin{smallmatrix}A^* & 0\\ 0 &A\end{smallmatrix}\right)\right)
    \det(A)^k\, A\cdot C_F(f).
  \end{equation}
  for all $A$ in $\GL$, such that $\left(\begin{smallmatrix}A^* & 0\\
      0 &A\end{smallmatrix}\right) \in \Gamma$.
\end{proposition}

This allows us to define a formal Siegel modular form.

\begin{definition}
  \label{def:fsmfs}
  Let $R$ be a module (or ring) with a left $\GL$ action, and let
  $\chi$ be a linear character of $\GL$.
  A formal Siegel modular form with values in $R$ and (formal) character $\chi$
  is a map $C:Q\to R$ for which
  \begin{equation}
      \label{eq:formal}
  C(A\cdot f) = \chi(A)\, A\cdot C(f)
  \end{equation}
  for all $A\in\GL$.
\end{definition}

By Proposition~\ref{prop:conclusion}, we see that this definition
models a Siegel modular form well, with the formal character
given by
\[
   \chi(A) = 
    \psi\left(\left(\begin{smallmatrix}A^* & 0\\ 0 &A\end{smallmatrix}\right)\right)
    \det(A)^k.
\]
Moreover, it is enough to store the
values of the coefficient map $C$ at the $\GL$-reduced elements of
$Q$, since the other coefficients can be computed
using~\eqref{eq:formal}.


We note that there are only 4 possibilities for the character of a
formal Siegel modular form:
\[
    \chi \in \set{ 1, \det, \sigma, \det\sigma}
\]
where $\sigma(A)$ is the sign of $A\bmod 2$ acting as a permutation on
the set $\set{(0,1),(1,0),(1,1)}$.
In fact, the group of linear characters of $\GL$ is non-cyclic of
order $4$ generated by $\det$ and $\sigma$.
This follows from the fact that the abelianization of $\GL$ is
isomorphic to the Klein four group.  Indeed, $\GL$ is generated by the
matrices $E=\smat{1&{}\\{}&-1}$, $T=\smat {1&1\\{}&1}$ and $S=\smat
{{}&{-1}\\1&{}}$. From $(ST)^3=S^2$ and $ETE=T^{-1}$ we deduce
$\pi(S)=\pi(T)^{-3}$ and $\pi(T)=\pi(T)^{-1}$, where $\pi$ is the
canonical map onto the abelianization of $\GL$. Therefore, the
abelianization is generated by $\pi(E)$ and $\pi(T)$, both of which
have order $1$ or $2$. On the other hand, since we have four
characters, the order of the abelianization is at least four. This
implies our claim.

For a Siegel modular form of level $\ell$, the formal character
can only be either $1$ or $\det$, depending only on the parity of the
weight $k$ and the parity of the character $\tilde\psi$.



\subsection{Formal Constructions}

In order to highlight the usefulness of thinking of Siegel modular
forms as purely formal objects, we now describe two analytic constructions from
this formal point of view.    

\subsubsection{Siegel $\Phi$ operator}  We now describe the Siegel $\Phi$ operator formally.  The
map $\Phi$ takes a degree $n$ Siegel modular form to a degree $n-1$
Siegel modular form.  In particular, cusp forms are in the kernel of
$\Phi$ and Eisenstein series are mapped to Eisenstein series.  The Siegel $\Phi$ operator for degree 2 Siegel modular forms has a very simple formula in terms of its
singular Fourier coefficients:
\begin{equation}
    \Phi(F) := \sum_{n\geq 0} C_F([0,0,n])\,q^n\,.
\end{equation}



\subsubsection{Hecke Operators}
In this section we give formulas for the image of a Siegel
modular form of level $N$ and weight $(k,j)$.  We give two
formulas.  The first is for the operator $T(p^\delta)$ when $p$ and $N$ are
coprime.  The second is for the operator $U(p)$ when $p$ divides the
level.

Let $F$ be a Siegel modular form as above.  Let $C$ be its
corresponding formal Siegel modular form and let $C'$ be the formal
Siegel modular form that corresponds to the image of $F$ under
$T(p^\delta)$. Then, as found in \cite{IbukiyamaHecke},
\[
C'([a,b,c])=\sum_{\alpha+\beta+\gamma=\delta}p^{\beta(k+j-2)+\gamma(2k+j-3)}
\sum_{\substack{U\in R(p^\beta)\\a_U \equiv
    0\,(p^{\beta+\gamma})\\ b_U \equiv c_U\equiv 0\,(p^{\gamma})}} (d_{0,\beta}U) \cdot C\left(p^{\alpha}\left[\frac{a_U}{p^{\beta+\gamma}},\frac{b_U}{p^\gamma},\frac{c_U}{p^{\gamma-\beta}}\right]\right)
\]
where $R(p^\beta)$ is a complete set of representatives for
$\SL/\Gamma^{(1)}_0(p^\beta)$; where, for $f=[a,b,c]$,
$[a_U,b_U,c_U]=f_U:=f\left((X,Y){}^tU\right)$; where $d_{0,\beta} =
\left(\begin{smallmatrix}1 & \\ & p^\beta\end{smallmatrix}\right)$;
and where the dot denotes the action of $d_{0,\beta}U$ as defined
in~\eqref{eq:action}.
We note that if $F$ is scalar-valued this action is trivial.

For primes $p$ dividing the level of a scalar-valued Siegel modular
form $F$, there is an operator $U(p)$.  In
\cite{Bocherer2}, B\"ocherer gives an easy explicit formula for the
action of $U(p)$ on a Siegel modular form.  In particular,
let $C$ and $C'$ be formal Siegel modular forms corresponding to $F$
and the image of $F$ under $U(p)$.  Then
\[
C'[(a,b,c)] = C[(pa,pb,pc)].
\]

\section{Recipes and Algorithms}\label{sec:algorithms}

The generators of rings of Siegel modular forms of degree 2 are often
given in one of the following ways: a linear combination of theta
series, a linear combination of products of theta constants, or, if
the form is of level 1, it may be given in terms of the Igusa
generators.  Additionally, particular Siegel modular forms of degree 2
can also be written in terms of Maass lifts, Eisenstein series, or
Borcherds products.  In this section, we present the definitions of
all these terms and algorithms or recipes that can be used to compute
them (or, at the very least, provide references to such methods). 

\subsection{Eisenstein Series}
\label{subsec:Eisenstein-series}
We define an Eisenstein series of degree 2 by
\[
E_k(Z) = \sum_{\{C,D\}} \det(CZ+D)^{-k}
\]
where the sum is over nonassociated symmetric coprime matrices $C,D\in
\Mat{2}(\Z)$ (here, two integral matrices $A$ and $B$ are called
coprime if the matrix products $GA$ and $GB$ are integral when and
only when $G$ is integral).  Earlier work by Maass \cite{Maass64} and
Mizumoto \cite{Mizumoto81,Mizumoto84} give explicit formulas for the
Fourier coefficients of Eisenstein series of level 1.  We describe
another approach here.

In \cite{EichlerZagier} one can find a
formula for the Fourier coefficients of a Siegel Eisenstein series of degree 2.
The formula is in
terms of Cohen's function which we define now.  Let
\[
L_D(s) = 
\begin{cases}
0 & \text{if }D\not\equiv 0,1\pmod{4}\\
\zeta(2s-1) &\text{if } D=0 \\
L(s,\kro{D_0}{\cdot})\sum_{d\isdiv
  f}\mu(d)\kro{D_0}{d}d^{-s}\sigma_{1-2s}(f/d) &\text{if }
D\equiv0,1 \pmod{4}\text{ and }D\neq 0\\
\end{cases}
\]
where $D$ has been written as $D=D_0f^2$ with $f\in \N$, $D_0$ is the
discriminant of $\Q(\sqrt{D})$ and $\sigma_k(n)=\sum_{d\isdiv n}
d^k$.  We define Cohen's function by $H(k-1,N) = L_{-N}(2-k)$.  

Now suppose $C_k$ is the formal Siegel modular form corresponding to
the Eisenstein series $E_k$.  Then, according to \cite[Cor. 2, p. 80]{EichlerZagier}, we have
\[
C_k([a,b,c]) = \sum_{d\isdiv (a,b,c)}d^{k-1}H\left(k-1,\frac{4ac-b^2}{d^2}\right).
\]
We observe that the coefficient of a singular form $C_k([0,0,c])$ is
$\zeta(3-2k)\sigma_k(c)$, which is the $c$th coefficient of the
Eisenstein series of degree 1 and weight $2k-2$.  We note that one can
compute Cohen's function within Sage.  

We remark that the weight $k$ Eisenstein series can also be computed
as the Saito-Kurokawa lift of the degree 1, weight $2k-2$ Eisenstein
series.  In fact, that's where the formulas in this section come
from.  See Section~\ref{ssec:SK} for more details on how to compute
Saito-Kurokawa lifts.

One can also consider Eisenstein series with level.  Mizuno
\cite[Theorem 1]{Mizuno} gives a very complicated formula for
Eisenstein series of degree 2 and odd squarefree level.  The formulas
consists of special values of L-functions, Gauss sums, and the like.
We refer the interested reader to that paper for details.

\subsection{Klingen-Eisenstein Series}

Given a modular form $f$ of
degree $n_1 \le n$ and  weight $k > n + n_1 + 1$ the
Klingen-Eisenstein series of $f$ is a modular
form $E_{n, n_1}(f)$ of degree $n$ and weight $k$ such that the Siegel $\Phi$
operator maps $E_{n, n_1}(f)$ to $f$.  More precisely, it is given by
the summation
\begin{align}
\label{eq:klingeneisenstein_def}
  E_{n,n_1}(f)
& =
  \sum_{M \in \Gamma_{n_1,\infty}^{(n)} \backslash \Gamma^{(n)}} f |M
\text{,}
\end{align}
where $f$ is considered a function on $\PH_n \hookleftarrow \PH_{n_1}$ and $\Gamma_{n_1, \infty}^{(n)}$ is the $n_1$-parabolic subgroup
\begin{align}\label{eq:parabolic subgroup}
  \Gamma_{n_1,\infty}^{(n)}
& := \{ \left(\begin{smallmatrix} A & B \\ C & D\end{smallmatrix}\right) \in \Gamma^{(n)} \;:\;
        C_{ij} = 0 \text{ for }i > n_1, \text{ all }j,\;
        D_{ij} = 0 \text{ for }i > n_1,\, j \le n_1\}
\text{,}
\end{align}
where the $ij$th entry of a matrix $M$ is denoted $M_{ij}$.

We restrict to $n = 2$.  Then $E_{2,0}(f) = E_2$ and $E_{2,2}(f)=f$ as
can be seen from the easy way in which the corresponding
$n_1$-parabolic subgroups can be described.  So
we have to give formulas for $E_{2,1}$ and this was done by Mizumoto in
\cite{Mizumoto81,Mizumoto84}.  We can restrict to the case that
$f$ is an eigenform.  Then, we can assume that $f$ is cuspidal, since
otherwise $E_{2,1}(f)$ is a classical Eisenstein series.

Mizumoto gives an explicit formula for the primitive coefficients of
a Klingen-Eisenstein series in terms of the coefficients of $f$ and
special values of three different L-functions.  We do not present this
explicit formula, but note, instead, that one can easily identify the
Klingen Eisenstein series among a basis of eigenforms.  And this is
how we compute their Fourier expansions in practice.

\subsection{Maass Lifts}\label{ssec:SK} This lift can be
thought of as a modular form that can, in some sense, be computed
easily from modular forms that are easier to compute.  In particular,
it can be derived from Jacobi forms and elliptic modular forms.  The
standard reference for Jacobi forms is \cite{EichlerZagier}.

\begin{definition}
  A Jacobi form of level 1, weight $k$ and index $m$ is a function
  $\phi(z,\tau)$ for $z\in \PH_1$ and $\tau\in \C$ such that:
  \begin{enumerate}
  \item
    $\phi\left(\frac{az+b}{cz+d},\frac{\tau}{cz+d}\right)=(cz+d)^ke^{\frac{2\pi
        i m c \tau^2}{cz+d}}\phi(\tau,z)$ for
    $\left(\begin{smallmatrix} a&
        b\\c&d\end{smallmatrix}\right)\in\SL$;
  \item $\phi(z,\tau+ \lambda z+\mu)=e^{-2\pi i m(\lambda^2z+2\lambda
      \tau)}\phi(z,\tau)$ for all integers $\lambda,\mu$; and
  \item $\phi$ has a Fourier expansion
    \[
    \phi(z,\tau)=\sum_{n\geq 0}\sum_{r^2\leq 4nm} d(n,r)q^n\zeta^r.
    \]
  \end{enumerate}
\end{definition}

A Maass lift \cite{Maass} is a lifting from the space of Jacobi forms
of weight $k$ and index 1 to the space of Siegel modular forms of
degree 2.

\begin{definition} \label{def:maass} Let $\phi$ be a Jacobi cusp form
  of weight $k$ and index $1$ and suppose
  $\phi(\tau,z)=\sum_{n>0,r\in\Z}d(n,r)q^n\zeta^r$ is its Fourier
  expansion.  Then the Maass lift of $\phi$ is $\text{Maass}(\phi)$ a
  Siegel modular form of degree 2 and weight $k$ given by
  \begin{equation*}
    \text{Maass}(\phi)\left(\begin{smallmatrix} \tau & z\\z& w\end{smallmatrix}\right) = \sum_{a,b,c}\left(\sum_{\delta|(a,b,c)}\delta^{k-1}d\left(\frac{ac}{\delta^2},\frac{b}{\delta}\right)\right)q^{a}\zeta^b q'^c.
  \end{equation*}
\end{definition}
We observe that $\text{Maass}(\phi)$ corresponds to a formal Siegel
modular form $C$ and we compute $C$ coefficient by coefficient:
\[
C_\phi([a,b,c]) =
\sum_{\delta\isdiv(a,b,c)}\delta^{k-1}d\left(\frac{ac}{\delta^2},\frac{b}{\delta}\right),
\]
where $d(n,r)$ are the coefficients of the Jacobi form $\phi$.

A map from elliptic modular forms of level 1 to Jacobi forms can be found in
\cite{Skoruppa}; we note that this is not the Saito-Kurokawa lift as
it is not Hecke invariant.  We summarize the result here.  Let
$\eta:=q^{1/24}\prod_{n=1}^\infty (1-q^n)$ and
\begin{align*}
  A&:=\eta^{-6}\sum_{r,s\in \Z, r\not\equiv s\pmod 2} s^2(-1)^rq^{(s^2+r^2)/4}\zeta^r,\text{ and}\\
  B&:=\eta^{-6}\sum_{r,s\in \Z, r\not\equiv s\pmod 2}
  (-1)^rq^{(s^2+r^2)/4}\zeta^r\\.
\end{align*}
For $f$ a degree 1, weight $k$ modular form and $g$ and degree 1 and
weight $k+2$ cusp form, the map, $I: (f,g)\mapsto
\frac{k}{2}fA-(q\frac{d}{dq}f)B+gB$ is a Jacobi form of weight $k$ and
index 1.  

The composition $\text{Maass}\circ I$ yields explicit formulas for the
Fourier coefficients of the forms in the Maass spezialschaar, i.e.~in
the image of the Maass lift, in terms of the Fourier coefficients of
elliptic modular forms of level~$1$. The formulas for the Fourier
coefficients of the Siegel Eisenstein series in
Section~\ref{subsec:Eisenstein-series} are special cases of this
construction.

\subsection{Theta Series}\label{ssec:theta-series}

Let $q$ be a positive definite quadratic form on a lattice $L$ of rank
$n$ and $b(v,w):=q(v+w)-q(v)-q(w)$ be its associated bilinear form. We
assume $q(L)\subseteq\Z$.  The theta series of degree $2$ associated
to $L$ can be defined as
\begin{equation*}
  \theta^{(2)}_L(Z) := \sum_{(v,w)\in L\times L}\, e\left(\tr (Z
    M_{[v,w]})\right)
\end{equation*}
where $M_{[v,w]} := M_f$ for the binary quadratic form $f = q(xv +
yw)$.  It is well known that this series is a (scalar valued) Siegel
modular form of degree 2 and weight $n/2$, for the level of the
lattice and a quadratic character related to the determinant
of the bilinear form.

The Fourier coefficients of
$\theta^{(2)}_L$ are related to representation numbers of binary
quadratic forms by $q$. Namely,
\begin{equation}
  \label{eqn:theta2coeff}
  C(f) = \#\left\{(v,w)\in L\times L \;:\; f=q(xv+yw)\right\}
  .
\end{equation}
The standard properties of representation numbers imply that $C(A\cdot f) =
C(f)$ for $A\in\GL$.  So, in particular, a theta series is a formal
Siegel modular form with trivial formal character.
 
The following algorithm computes the Fourier coefficients $C(f)$ of a
theta series of degree 2, for $\GL$-reduced binary quadratic forms
$f$.

\begin{algorithm}\label{alg:thetaseries}
  Compute the Fourier coefficients $C(f)$ of $\theta^{(2)}_L$ for $0 >
  \disc f \geq -X$ as well as $C([0,0,c])$ for $c\leq c_{max}$ where $c_{max}$ is $\left\lfloor
    \frac{X+1}{4}\right\rfloor$.
  \begin{enumerate}
  \item[1.] [Precompute] For $1\leq i\leq c_{max}$, compute $V_i := \{
    v\in L \;:\; q(v) = i\}$ and $\tilde{V}_i\subset V_i$ such that
    for every vector $v\in V_i$, exactly one of $v$ or $-v$ is in
    $\tilde{V}_i$.
  \item[2.] [Compute the singular coefficients] $C(\qf 000) = 1$ and
    for $1\leq i\leq c_{max}$ set $C(\qf 00i) = 2\cdot\#\tilde{V}_i$ .
  \item[3.] [Iterate through all $\GL$-reduced forms] For each $1 \leq
    a \leq \left\lfloor \sqrt{\frac{X}{3}} \right\rfloor$ and for each
    $a\leq c \leq \left\lfloor \frac{a^2+X}{4a}\right\rfloor$, look up
    $V_a$ and $V_c$.
  \item[4.] [Iterate through vectors of shortest length] For $v\in
    V_a$, compute the linear form $l(w) = b(v,w)$.  For $w \in V_c$,
    compute $b = l(w)$ and $\lambda(b) := \begin{cases}
      2 & \text{if $b\neq 0$,} \\
      4 & \text{if $b=0$.}
    \end{cases}$
  \item[5.] [Compute coefficient $C[a,b,c])$] If $b\leq a$ and
    $b^2-4ac\geq -X$, then $C[(a,b,c)]=C[(a,b,c)] + \lambda(b)$.
  \item[6.] [Done] Return the map $C(f)$.
  \end{enumerate}

\end{algorithm}
The algorithm is implemented in Sage a little differently (in
particular the loops are in a different order in order to avoid
computing things we do not really need to compute).

\subsection{Theta constants}\label{ssec:theta-constants}  As we will see below in
Section~\ref{sec:rings}, being able to compute the products of theta constants is a very useful way to compute particular modular forms.
For $a,b\in \{0,1\}^2$ (column
vectors), define the theta constant $\theta_{a,b}$ as
\begin{equation}\label{eq:thetac}
 \theta_{a,b}(Z)
 =
 \sum_{\begin{subarray}{c}\ell\in\Z^2\\ l\equiv a\bmod 2
   \end{subarray}
 }e\left(\frac{1}{8} Z[\ell]+\frac{1}{4}{}^t\ell b\right)
\end{equation}
where we use $A[B] = {}^tBAB$. It is easy to see that~$\theta_{a,b}$
is identically zero if ${}^ta b$ is odd.  We note that an individual
theta constant is not, in general, a formal Siegel modular form as can be seen from
the following lemma:
\begin{lemma}
For $A\in\GL$, $f \in Q$, we have:
\[
C_{a,b} ( A\cdot f) = C_{A^{-1}a,{}^tAb} (f).
\]
\end{lemma}
Note for example, that $\theta_{0000}(Z)$ and
$\theta_{0010}(Z)+\theta_{0001}(Z)+\theta_{0011}(Z)$ are formal Siegel
modular forms but $\theta_{0010}(Z)$ is not.

We have the following result that allows for the computation of the
coefficients of a single theta constant.  As noted above, this form is
not necessarily a formal Siegel modular form.
\begin{proposition}
  \label{prop:theta-constant}
  Assume that ${}^ta b$ is even. Then the map $C_{a,b}:Q\to \Z$ given
  by
  \begin{equation*}
    C_{a,b}(f)
    =
    \sum_{\substack{\ell \equiv a \bmod{2}\\ \ell^2 = f}}(-1)^{\ell(b)/2}
  \end{equation*}
  describes the coefficients of $\theta_{a,b}(8Z)$.  Here $Q$ is as defined in
  Section~\ref{ssec:formal}, $a={}^t(a_1,a_2)$ is identified with the
  linear polynomial $a_1X+a_2Y$ and the sum is over all linear forms
  $\ell$ in $\Z[X,Y]$ which are congruent to $a_1X+a_2Y$ mod 2 and
  whose square equals $f$.
\end{proposition}
The claimed formula becomes obvious if one writes
$\ell={}^t(\ell_1,\ell_2)$ and $Z[\ell]=\ell_1^2\tau+2\ell_1\ell_2z+\ell_2^2\tau'$.

Particular Siegel modular forms are often described as an algebraic
combination of theta constants.  The following proposition allows for
the computation of the Fourier coefficients of  such forms.
\begin{proposition}\label{prop:products-theta-constants}
  Let
  \begin{equation*}
    F(Z)
    :=
    \sum_{i=1}^n\alpha_i \prod_{j=1}^d \theta_{a_{ij},b_{ij}}(Z),
  \end{equation*}
  where we assume that all ${}^ta_{ij}b_{ij}$ are even.  Then the
  map $C:Q\to \Z$ associated to $F(8Z)$ is given by
  \begin{equation*}
    C(f)
    =
    \sum_{i=1}^{n-1}\alpha_i
    \sum_{\substack{\ell_1,\dots,\ell_d\equiv a_{i1},\dots,a_{id}\bmod{2}\\\ell_1^2+\cdots+\ell_d^2 = f}}
    (-1)^{(\ell_1(b_{i1})+\cdots+\ell_d(b_{id}))/2}.
  \end{equation*}
\end{proposition}

The $C$ in the previous two propositions can be formal Siegel modular
forms.  In that case we only compute the values of $C$ at
$\GL$-reduced quadratic forms in $Q$, taking advantage of the extra
structure a formal Siegel modular form has.  We note, also, that
whether or not $C$ is a formal Siegel modular form has to be
determined before the computations are carried out.

The above definitions and propositions can be generalized to other
levels.  One should change these results in the expected way:  e.g.,
the sums are mod $N$, the forms $\theta_{a,b,N}(2N^2Z)$ are formal
Siegel modular forms, etc. In particular, we write
\begin{equation}\label{eq:thetac-levelN}
  \theta_{N,a,b}(Z)
  =
  \sum_{\begin{subarray}{c}\ell\in\Z^2\\ \ell\equiv a\bmod N
    \end{subarray}
  }e\left(\frac{1}{2 N^2} Z[\ell]+\frac{1}{N^2}{}^t\ell b\right),
\end{equation}
where again we use $A[B] = {}^tBAB$. 

\begin{remark}
A useful observation is that one can always take $b=0$ to get the
space spanned by theta constants.  In particular, the span of
$\theta_{a,b,N}(Z)$ is the same as the span of $\theta_{a,0,N}(Z)$.
\end{remark}


\subsection{Vector-valued theta series}
\label{ssec:vvthetaseries}

Following \cite{IbukiyamaSym46} we can define vector-valued theta
series.  Take a rank $n$ even unimodular lattice $L$, let
$P\in\C[X,Y]_j$ be pluriharmonic (i.e., for every $A\in\GL$,
$P(AX)=\det(A)^k\textrm{Sym}_j(A)P(X)$).  Then we define
\[
\theta_{L,P}(Z)
=
\sum_{x,y\in
  L}P\left(\begin{smallmatrix}x\\y\end{smallmatrix}\right)e^{\pi i
    ((x,x)\tau+(x,y)z+(y,y)\tau')}.
\]
One can show that $\theta_{L,P}\in M_{n/2+k,j}(\Gamma)$.  After
collecting terms, $\theta_{L,P}$ is evidently a formal Siegel modular
form because of the invariance of the pluriharmonic polynomials that
make up the coefficients.

To describe the generators of a certain space of vector-valued Siegel
modular forms, it is useful to define the following two pluriharmonic
polynomials.  Let $a,b\in \C^n$ be such that their inner products
$(a,a)=(a,b)=(b,b)=0$ (here the inner product of $a$ and $b$ is given
by $\sum a_ib_i$, without complex conjugation).  Consider the two polynomials
\begin{align*}
P_a^{(j)} &= \sum_{\nu = 0}^k(x,a)^{(j-\nu)}(y,a)^\nu X^{\nu}Y^{j-\nu}\\
P_{a,b,k}^{(j)} &= \sum_{\nu = 0}^k(x,a)^{(j-\nu)}(y,a)^\nu \det\begin{pmatrix}(x,a)&(y,a)\\(x,b)&(y,b)\end{pmatrix}^kX^{\nu}Y^{j-\nu}.
\end{align*}
With this notation, $\theta_{oL,P_a^{j} }\in M_{d/2,j}(\Gamma)$,
  $\theta_{L,P_{a,b,k}^{(j)}}\in M_{d/2+k,j}(\Gamma)$.

In order to compute these theta series, one needs to adjust
Algorithm~\ref{alg:thetaseries} appropriately. 




\subsection{Wronskians}\label{ssec:wronskians}

We recall a useful construction in \cite{Ibukiyama}.  A Siegel modular
form of degree 2 is a function in three variables, $\tau$, $z$ and
$\tau'$.  Suppose we have 4 algebraically independent Siegel modular
forms $F_1,F_2,F_3,F_4$ of weights $k_1,k_2,k_3,k_4$, respectively.
Then the function
\begin{equation}\label{eq:wronskian}
  \{F_1,F_2,F_3,F_4\}:=
  \left |
    \begin{matrix}
      k_1F_1& k_2F_2 & k_3F_3 &k_4F_4\\
      \frac{\partial F_1}{\partial \tau} &\frac{\partial F_2}{\partial \tau} &\frac{\partial F_3}{\partial \tau} &\frac{\partial F_4}{\partial \tau} \\
      \frac{\partial F_1}{\partial z} &\frac{\partial F_2}{\partial z} &\frac{\partial F_3}{\partial z} &\frac{\partial F_4}{\partial z} \\
      \frac{\partial F_1}{\partial \tau'} &\frac{\partial F_2}{\partial \tau'} &\frac{\partial F_3}{\partial \tau'} &\frac{\partial F_4}{\partial \tau'} \\
    \end{matrix}
  \right |
\end{equation}
is a nonzero Siegel cusp form of weight $k_1+k_2+k_3+k_4+3$.  See
\cite[Prop. 2.1]{Ibukiyama} for more details and for the statement for
arbitrary degree.

Computationally speaking, computing the determinant
$\{F_1,\dots,F_4\}$ in the straightforward way is very expensive.
Many products (36, computing the determinant in the naive way) of
multivariate power series would need to be computed in order to find
$F$. However, using basic properties of the determinant, one deduces
\begin{proposition}
  \label{prop:chi35}
  With notation as above, we have
  \begin{multline}
    \label{eq:fastchi35}
    F=\{F_1,F_2,F_3,F_4\}= \sum_{\substack{T_1 = [a_1,b_1,c_1]\geq 0\\
        T_2 = [a_2,b_2,c_2]\geq 0\\
        T_3 = [a_3,b_3,c_3]> 0\\
        T_4 = [a_4,b_4,c_4]> 0 }}
    C_{F_1}(T_1)C_{F_2}(T_2)C_{F_3}(T_3)C_{F_4}(T_4)\\
    \times \left | \begin{matrix}k_1&k_2&k_3&k_4\\ a_1&a_2&a_3&a_4\\
        b_1&b_2&b_3&b_4\\ c_1&c_2&c_3&c_4\end{matrix}\right |e(\sum
    a_i \tau + \sum b_i z +\sum c_i \tau').
  \end{multline}
\end{proposition}

This is the basis for Algorithm~\ref{alg:wronskian}, below.  The subtle part
in the algorithm is determining the bounds of the coefficients of the
$T_i$.  These bounds have to be as tight as possible, since even a
slight inaccuracy leads to subsequent iteration of many unneccessary
Fourier indices.
\begin{proof}[of Proposition~\ref{prop:chi35}]
   Equation~\eqref{eq:fastchi35} follows from writing~\eqref{eq:chi35}
   in the form
  \begin{multline*}
    F(Z) =
    \det\big(\sum_{T_1=[a_1,b_1,c_1]}C_{F_1}(T_1)\left(\begin{smallmatrix}
        k_1\\ a_1\\ b_1\\ c_1\end{smallmatrix} \right)
    e(a_1\tau+b_1z+c_1\tau'),\\
    \sum_{T_2=[a_2,b_2,c_2]}C_{F_2}(T_2)\left(\begin{smallmatrix} k_2\\
        a_2\\ b2\\ c_2\end{smallmatrix} \right)
    e(a_2\tau+b_2z+c_2\tau'),\dots \big)
  \end{multline*}
  and then using that the determinant function is multilinear.
\end{proof}

The algorithm below to compute $\{F_1,\dots,F_4\}$ can be changed in
the obvious way to work for arbitrary degree.
\begin{algorithm}\label{alg:wronskian}
  Compute the Fourier coefficients $C(f)$ of $\{F_1,\dots,F_4\}$ for
  $0 > \disc f \geq -X$ as well as $C([0,0,c])$ where $c\leq\lfloor\frac{X+1}{4}\rfloor$.
  \begin{enumerate}
  \item[1.]  [Precompute quadratic forms and modular forms] Compute
    $Q_X:= \{f\in Q:0 \geq \disc f \geq -X\}\cup \{f = [0,0,c]\in Q: c\leq\lfloor\frac{X+1}{4}\rfloor\}$; compute coefficients
    $C_{F_1}(f), C_{F_2}(f),C_{F_3}(f),C_{F_4}(f)$ for all $f\in Q_X$.
  \item[2.] [Compute the coefficient at $f$] For each 4-tuple
    $(T_1,T_2,T_3,T_4)$ where $T_i\in Q_X$ and $\sum T_i = f$, set
    \[
    C(f) =  \sum_{\substack{T_1 = [a_1,b_1,c_1]\geq 0\\
        T_2 = [a_2,b_2,c_2]\geq 0\\
        T_3 = [a_3,b_3,c_3]> 0\\
        T_4 = [a_4,b_4,c_4]> 0 }}
    C_{F_1}(T_1)C_{F_2}(T_2)C_{F_3}(T_3)C_{F_4}(T_4)\\
    \times \left | \begin{matrix}k_1&k_2&k_3&k_4\\ a_1&a_2&a_3&a_4\\
        b_1&b_2&b_3&b_4\\ c_1&c_2&c_3&c_4\end{matrix}\right |.
    \]
  \item[3.] [Done] Return the map $C(f)$.
  \end{enumerate}
\end{algorithm}

\subsection{Rankin-Cohen Brackets}

The attempt to construct a modular form as a linear combination of
products of derivatives of two modular forms is known as the
Rankin-Cohen construction.  
A complete investigation of this and related constructions has been
given by Ibukiyama in \cite{IbukiyamaDifferential}.  See also
\cite{ChoieEholzer,EholzerIbukiyama,Miyawaki} for characterizations of
various Rankin-Cohen brackets.

\subsubsection{Satoh Bracket}\label{ssec:satoh-bracket}
The Satoh bracket is a special case of
this general construction.  Satoh examined the case $\rho = \rsym_2
\otimes \det^{k}$.  Suppose $F\in M_{k}(\Gamma)$ and $G\in M_{k'}(\Gamma)$.
Then 
\begin{align}
  \label{eq:satoh-bracket}
  [F, G]_{\rsym_2} & = \frac{1}{2 \pi i} \bigl(\frac{1}{k} G\, \partial_Z F
  - \frac{1}{k'} F\, \partial_Z G\bigr) \in M_{k+k',2}(\Gamma)\text{,}
\end{align}
where $\partial_Z = \left(\begin{smallmatrix} \partial_{Z_{11}} &
    1/2\, \partial_{Z_{12}} \\ 1/2\, \partial_{Z_{12}}
    & \partial_{Z_{22}}\end{smallmatrix}\right)$.

In order to implement the Satoh bracket in our package, we need first
verify that the forms $\partial_ZF$ and $\partial_ZG$ are formal
Siegel modular forms, even if they are not Siegel modular forms.

\begin{lemma}\label{lem:bracket}
Let $F$ be a modular form.  Then $\partial_Z F$ as defined above is a
formal Siegel modular form.
\end{lemma}
\begin{proof}
Let $F_Z:=\partial_Z F$ and $C_{F_Z}:Q\to \C[X,Y]_j$ be
such that $C_{F_Z}([a,b,c])$ takes the coefficient of $F_Z$ at
$[a,b,c]$ as its value.  We need to show that $C_{F_Z}$ is
formal;~i.e., that $C_{F_Z}(A\cdot [a,b,c])=(\det A)^k A \cdot C_{F_Z}([a,b,c])$.

A straightforward calculation shows that
$C_{F_Z}([(a,b,c)])=C_F([a,b,c])(aX^2+bXY+cY^2)$.  Now
\begin{align*}
C_{F_Z}(A\cdot [a,b,c])  &= C_F(A\cdot[a,b,c])A\cdot(aX^2+bXY+cY^2)\\
&= (\det A)^kC_F([a,b,c])A\cdot(aX^2+bXY+cY^2)\\
&= (\det A)^kA\cdot (C_F([a,b,c])(aX^2+bXY+cY^2)\\
&= (\det A)^k A\cdot C_{F_Z}([a,b,c]).
\end{align*}
\end{proof}

Because of Lemma~\ref{lem:bracket}, we can compute the Satoh bracket
as the difference of the product of two formal Siegel modular forms.

Assuming that multiplication of vector valued Fourier expansions is implemented an algorithm can now be formulated as follows.
\begin{algorithm}
\label{alg:satoh-bracket}

\begin{enumerate}
\item[1.] For $f$ and $g$ compute the Fourier expansion with coefficients $C_{\partial_ZF}(T) = C(T) (X^2 T_{11} + X Y T_{12} + Y^2 T_{22})$.
\item[2.] Multiply the Fourier expansions according to formula \eqref{eq:satoh-bracket}.
\end{enumerate}
\end{algorithm}

\subsubsection{Ibukiyama's Brackets}

There are constructions found in \cite{IbukiyamaSym46} that are analogous
to the Satoh bracket but whose images are in $M_{k,4}(\Gamma)$ and
$M_{k,6}(\Gamma)$.  To give an indication of what they look like, we
provide an example: let $F\in M_k(\Gamma)$ and $G\in M_{k'}(\Gamma)$.
Then if $k$ is even,
\begin{multline*}
[F,G]_{\rsym_4}=\frac{k'(k'+1)}{2}G\times 
\begin{pmatrix}
\frac{\partial^2F}{\partial \tau^2}\\
2\frac{\partial^2F}{\partial \tau\partial z}\\
\frac{\partial^2F}{\partial z^2}+2\frac{\partial^2F}{\partial \tau\partial\tau'}\\
2\frac{\partial^2F}{\partial \tau'\partial z}\\
\frac{\partial^2F}{\partial \tau'^2}
\end{pmatrix}
-\\(k'+1)(k+1)\begin{pmatrix}
\frac{\partial F}{\partial \tau}\frac{\partial G}{\partial \tau}\\
\frac{\partial F}{\partial \tau}\frac{\partial G}{\partial  z}+\frac{\partial G}{\partial \tau}\frac{\partial F}{\partial z}\\
\frac{\partial F}{\partial \tau}\frac{\partial G}{\partial  \tau'}+\frac{\partial G}{\partial \tau}\frac{\partial F}{\partial
  \tau'}+\frac{\partial F}{\partial z}\frac{\partial G}{\partial z}\\
\frac{\partial F}{\partial \tau'}\frac{\partial G}{\partial  z}+\frac{\partial G}{\partial \tau'}\frac{\partial F}{\partial z}\\
\frac{\partial F}{\partial \tau'}\frac{\partial G}{\partial \tau'}
\end{pmatrix}+
\frac{k(k+1)}{2}\begin{pmatrix}
\frac{\partial^2G}{\partial \tau^2}\\
2\frac{\partial^2G}{\partial \tau\partial z}\\
\frac{\partial^2G}{\partial z^2}+2\frac{\partial^2F}{\partial \tau\partial\tau'}\\
2\frac{\partial^2G}{\partial \tau'\partial z}\\
\frac{\partial^2G}{\partial \tau'^2}
\end{pmatrix}.
\end{multline*}




\subsection{Borcherds Products}\label{ssec:borcherdsproducts}

In \cite{BorcherdsGrassmannians}, Borcherds introduced a
multiplicative lift from whose image can be a Siegel modular form of
degree~$2$.  The Fourier expansions of these Borcherds products can be
computed in polynomial time~\cite{GKR11}.  We refer the interested
reader to the latter article, co-authored by the first author, for
more details. 

\subsection{Further remarks}

\subsubsection{Restriction method}

We are going to discuss the approach that was used to obtain the
results in \cite{PoorYuen}.  In this work Siegel modular forms of
degree $4$ were considered, but in principle the algorithm can also be
used to obtain Fourier expansions of forms for high level and degree
$2$.  The idea is as follows.  Suppose we restrict a Siegel modular
form $f : \PH_2 \rightarrow \C$ to the subspace $\phi_{s,\zeta} :
\PH_1 \rightarrow \PH_2,\, \tau \mapsto s\tau + \zeta$ for some
positive definite symmetric $s \in M(2, \Z)$ and some symmetric $\zeta
\in M(2, \Q)$.  This will lead to an elliptic modular form whose Fourier expansion can be deduced from the Fourier expansion of $f$.  Since many cusps for $f \circ \phi_{s,\zeta}$ can be identified in $\PH_2$, there are non trivial obstructions to the Fourier expansion of $f$.  In \cite{PoorYuen}, where all cusps identified with each other, this provided enough restrictions to actually compute the initial Fourier expansions of a basis for a space of fixed weight.  They could not give a proof that this method will work for all Fourier indices and all weights.  In the case of level subgroups a further difficulty enters.  There will be cusps that cannot be identified with each other in $\PH_2$.

\subsubsection{Poor-Yuen Conjecture}

The Fourier-Jacobi expansion of a Siegel modular form can be used to
describe a Fourier expansion.  A formal Fourier-Jacobi expansion is a
formal series $\sum f_m q^m$ of Jacobi forms $f_m$ that have level
$m$.  Chris Poor and David Yuen conjectured that a formal
Fourier-Jacobi expansion is the Fourier-Jacobi expansion a
Siegel modular form for the full modular group if and only if the
Fourier coefficients of the Jacobi forms satisfy the condition $f_m(n,
r) = f_n(m,r)$.  It was recently announced by Poor and Yuen that have
proved this result.

\subsubsection{Invariant method and Igusa's theorem on theta constants}
\label{ssec:invariantmethods}

We fix a level $N$.  For $a,b\in \{0,\ldots,N-1\}^2$ (column
vectors), define the theta constant $\theta_{N,a,b}$ as in \eqref{eq:thetac-levelN}.
It is easy to see that~$\theta_{N,a,b}$
is identically zero if $a \slashdiv N$ and $b \slashdiv N$ have half integral entries and ${}^ta b$ is odd.  \cite{Igusa5} tells us that also the converse holds.  There we can also find a proof that they are modular forms with respect to the following subgroup $\Gamma(N, N)$ of $\Sp$: 
\[
\left\{\gamma=\left(\begin{smallmatrix}A&B\\C&D\end{smallmatrix}\right)\in \textrm{Sp}(2, \Z):\gamma\equiv 1_4\pmod{N};\,\,A{}^tB, C{}^tD \textrm{ have diagonals divisible by $2N$}\right\}.
\]
  As we will see below in Section~\ref{ssec:rings}, being able to compute the products of theta constants is a very useful way to describe modular forms.  

An important result is
\begin{theorem}
Suppose that $\Gamma(N,N) \subseteq \Gamma$ is a congruence subgroup.  The graded ring of modular forms for $\Gamma$ is the integral closure of the ring of those modular forms with respect to $\Gamma$ that are formed as linear combinations of products of the theta series $\theta_{N,a,b}$.
\end{theorem}
This method to compute rings of Siegel modular forms is impractical for large level.  Nevertheless its connection to coding theory makes it interesting for general level.  This can be found in \cite{Runge}.

\subsubsection{The basis problem and theta series}\label{sec:basis problem}


The basis problem, as formulated by Eichler
\cite{Eichler}\cite{EichlerCorrection}, asks whether the space of
elliptic modular forms for fixed weight, level and character can be
spanned by $\theta$-series.  The affirmative answer in the elliptic
case partially generalizes to Siegel modular forms.  The most complete
treatment can be found in \cite{BoechererBasis}, where B\"ocherer
proves the affirmative answer in the case of squarefree level and also
gives a brief survey on the current effort in this area.  In
particular, he proves the following:
\begin{theorem}
Assume that $N \in \N$ is squarefree and $k > 5$.  Let $\chi$ be a
character for $\Gamma_0(N)$.  Then all cusp forms in $M_k(\Gamma_0 (N), \chi)$ are linear combinations of $\theta$-series. 
\end{theorem}
We also mention the very early treatment in \cite{Ozeki}.  There Ozeki proves that in the case of the full modular group one might restrict to unimodular lattices if and only if the weight is divisible by $4$.

There is a special case of this question.  The critical weight for degree $n$ forms is $n \slashdiv 2$; in the case of degree 2 the critical weight is $1$.  It is expected that these weights behave particularly well with respect to the basis problem.  In \cite{Weissauer} it was proved that every Siegel modular form of degree $2$ and weight $1$ is a linear combination of $\theta$-series.
At this point it might also be useful to the reader to mention the following results of the third named author.  In \cite{IbukiyamaSkoruppa} it was proved that there are no Siegel cusp forms of weight one for any group $\Gamma_0(N)$ of Hecke type.

Any modular form of weight less than the critical weight is called singular. In the case of degree 2 these are the half integral weight forms.  Singular weights are well-understood for arbitrary degrees (see \cite{Freitag}).  In particular, we know that all weight $\tfrac12$ forms associated to the full modular group only have nonvanishing Fourier coefficients for indefinite indices.  Hence, by a theorem of Serre and Stark \cite{SerreStark} they are given by $\theta$-series.

\section{Rings of Siegel Modular Forms}\label{sec:rings}

The problem of determining rings of modular forms has been of great interest.  Complete descriptions of some rings are presented in Section \ref{ssec:rings}.  These results are almost entirely based on a combination of the theory of Borcherds products (see \ref{ssec:borcherdsproducts}), $\theta$-series, and the invariant method that we discussed in Section \ref{ssec:invariantmethods}.

Besides these results, there are methods that have been developed for cases where a set of generators has not yet been obtained.  Although we focus on the descriptions mentioned above, we discuss some of these other methods in the next few subsections.  Indeed, even for moderate levels, rings of Siegel modular forms become complicated.  For applications it is sometimes sufficient to obtain a set of modular forms that spans a space of fixed weight.  In these cases it is necessary to know the dimensions of these spaces.  For this reason we also review known dimension formulas.

\subsection{Hecke eigenforms}

Hecke eigenforms are the modular forms that are of the greatest
arithmetic interest.  In order to compute eigenforms, one needs to
compute a space of cusp forms, find the matrix for a Hecke
operator and from that find the eigenforms.  In this section we
address issues surrounding this process.  First, one needs to know
when one has computed a basis for the space and to do this one needs to
know the dimension of the space.  Second, one needs to carry out the
linear algebra computations described above.

\subsubsection{Dimension formulas}

In most cases only formulas for dimensions of spaces of cusp forms are
available.  This might already be sufficient, since many arithmetic
questions can be reduced to this case.  In Table
\ref{tab:dimensionformulas} we give some dimensions of spaces of
Siegel modular forms.

\begin{table}
\begin{tabular}{l l}\hline\hline
$M_k(\Sp)$ & \cite{Igusa2,Skoruppa}\\
$S_k(\Gamma_1(\ell)), N\geq 3, k\geq 5 $ & \cite{Christian,Morita,Yamazaki}\\
$S_k(\Gamma_0(p)),\, p\text{ prime }, k\geq 5$ & \cite{Hashimoto}\\
$S_4(\Gamma_0(p)) ,\, p\leq 13 \text{ prime }$ & \cite{PoorYuenDimensions}\\
$S_3(\Gamma_0(p)) ,\, p\leq 23 \text{ prime }$ & \cite{PoorYuenDimensions}\\
$S_2(\Gamma_0(p)) ,\, p\leq 41\text{ prime }$ & \cite{PoorYuenDimensions}\\
$S_1(\Gamma_0(N)) ,\, N\geq 1$ & \cite{IbukiyamaSkoruppa}\\
$S_{k,j}(\Sp) ,\, j\geq 0, k\geq 4$& \cite{Tsushima}\\
\end{tabular}
\caption{Spaces and references for their dimension formulas}
\label{tab:dimensionformulas}
\end{table}

\subsubsection{Computing Hecke eigenforms}

Among all modular forms, the most arithmetically distinguished are the
Hecke eigenforms.  Fix a space of Siegel modular forms with basis
$\{F_1,\dots,F_n\}$.  Because the Hecke operators are a commuting
family of linear operators, there is a basis $\{G_1,\dots,G_n\}$ for
the space made up entirely of simultaneous eigenforms.

The forms $G_i$ are determined computationally as follows.  First,
determine the matrix representation for the Hecke operator $T_2$.  Do this by
computing the image under $T_2$ of each basis element $F_i$.  Build a
matrix $N$ that is invertible and whose $j$th row consists of coefficients of
$F_j$ at certain indices $Q_1,\dots,Q_n$.  To ensure that $N$ is
invertible we pick the indices one at time, making sure that each
choice of index $Q_i$ increases the rank of the matrix.  We then
construct a matrix $M$ whose $j$th row consists of coefficients of the
image of $F_j$ under $T_2$ indexed by $Q_1,\dots,Q_n$.  Then the
matrix representation of $T_2$ is $M\cdot N^{-1}$.  

Second, one finds the eigenspaces of the matrix for $T_2$ and uses
that to describe the basis $G_1,\dots,G_n$.  Typically only a single
member of a given Galois orbit is described.  See, for example, \cite{smf-tables}.

\subsection{Particular Rings of scalar-valued Forms}\label{ssec:rings}

\subsubsection{The full Siegel modular group}\label{sssec:fullgroup}  In \cite{Igusa2} and \cite{Igusa3}, the five generators of the ring of Siegel modular forms and level 1 are identified: they are $E_4$, $E_6$, $\chi_{10}$, $\chi_{12}$, $\chi_{35}$ and are of weights 4, 6, 10, 12 and 35.  The first two are Eisenstein series and the last three are cusp forms.  In \cite{Skoruppa} an indication of how to compute the four generators of even weight is given.  One composes an explicit map from classical to Jacobi forms and the Maass lift from Jacobi forms to Siegel modular forms of degree 2.  The implementation in \cite{smf-package} of these four generators follows \cite{Skoruppa}. 


In order to get the full ring of Siegel modular forms of level 1
without character, we require a method to compute $\chi_{35}$.  In
\cite{Igusa3}, it is given in terms of theta constants.  In
\cite{Gritsenko}, it is given as a Borcherds product.  In
\cite{Ibukiyama}, it is given as a Wronskian (see Section~\ref{ssec:wronskians}):
\begin{equation}
  \label{eq:chi35}
  \chi_{35} = \frac{1}{(2\pi i)^3}\{E_4,E_6,\chi_{10},\chi_{12}\}.
\end{equation}
Here $E_4,E_6,\chi_{10},\chi_{12}$ are normalized so that $C_{E_4}([0,0,0]) =
C_{E_6}([0,0,0]) = 1$ and $C_{\chi_{10}}([1,1,1]) =
C_{\chi_{12}}([1,1,1])=1$.

If one is content with just a basis for the vector space of Siegel modular forms of a fixed even weight, an observation by the first author is useful and would make, if true in general, the computation of such a basis significantly less expensive.  In particular, in \cite{Raum} he has observed that up to weight 172 one can span the space of weight $k$ Siegel modular forms of degree 2 via forms that are the sum of products of no more than two Saito-Kurokawa lifts; see Section~\ref{ssec:SK} for more details on Saito-Kurokawa lifts.  Since computing a Saito-Kurokawa lift is rather straightforward, multiplying a pair of Siegel modular forms is not terribly expensive, and adding many Siegel modular forms is very fast, this would represent a significant speed-up over what is done now.

\subsubsection{Congruence subgroups of small level}\label{ssec:level}

In papers by Ibukiyama and his coauthors, the generators of some rings of modular forms of small level have been identified.  Let $\Gamma_0(\ell)$ be the Klingen subgroup of level $\ell$.  For a subgroup $\Gamma$ we use $M_*(\Gamma,\chi)$ for the ring of Siegel modular forms of degree 2 with respect to the group $\Gamma$ and with character $\chi$ (we omit the $\chi$ if it is trivial).

The groups whose rings of modular forms are known to be described in
terms of generators are listed in Table \ref{tbl:subgroups} 
\begin{table}[h]
\begin{tabular}{ll}\hline\hline
Group and Character & Reference\\\hline
$\Gamma_0(2)$ & \cite{Ibukiyama}\\\hline
$\Gamma_0^{\psi_3}(3):=\{\gamma\in\Gamma_0(3): \psi_3(\gamma)=1\}$ where $\psi_3\left(\begin{smallmatrix}A&B\\C&D\end{smallmatrix}\right) := \kro{-3}{\det D}$, for all $\left(\begin{smallmatrix}A&B\\C&D\end{smallmatrix}\right)\in\Gamma_0(3)$ & \cite{Ibukiyama}\\\hline
$\Gamma_0^{\psi_4}(4):=\{\gamma\in\Gamma_0(4): \psi_4(\gamma)=1\}$ where 
$\psi_4\left(\begin{smallmatrix}A&B\\C&D\end{smallmatrix}\right) := \kro{-1}{\det D}$, for all $\left(\begin{smallmatrix}A&B\\C&D\end{smallmatrix}\right)\in\Gamma_0(4)$ & \cite{Ibukiyama3}\\\hline
\end{tabular}
\caption{Subgroups $\Gamma$ for which the generators of the ring $M_*(\Gamma)$ are known.}\label{tbl:subgroups}
\end{table}


\subsubsection{Siegel modular forms of degree 2 and level 2}
Consider the ring $M_*(\Gamma_0(2))$.  In \cite{Ibukiyama} it is shown that the ring has 5 generators:
\begin{align*}
X&:=(\theta_{(0,0),(0,0)}^4+\theta_{(0,0),(0,1)}^4+\theta_{(0,0),(1,0)}^4+\theta_{(0,0),(1,1)}^4)/4\\
Y &:= (\theta_{(0,0),(0,0)}\theta_{(0,0),(0,1)}\theta_{(0,0),(1,0)}\theta_{(0,0),(1,1)})^2\\
Z&:= (\theta_{(0,1),(0,0)}^4-\theta_{(0,1),(1,0)}^4)^2/16384\\
K&:=(\theta_{(0,1),(0,0)}\theta_{(0,1),(1,0)}\theta_{(1,0),(0,0)}\theta_{(1,0),(0,1)}\theta_{(1,1),(0,0)}\theta_{(1,1),(1,1)})^2/4096\\
\chi_{19} &= \{X,Y,Z,K\}
\end{align*}
The generators are, respectively, of weights 2, 4, 4, 6 and 19.  The first four are given in terms of theta constants and the last one is given as a Wronskian.

\subsubsection{Siegel modular forms of degree 2 and level 3}
The ring $M_*(\Gamma^{\psi(3)}_0(3))$ has 5 generators.  Consider the quadratic forms
\begin{align*}
A_2&:=\left(\begin{smallmatrix} 
2&1\\1&2
\end{smallmatrix}\right),\\
E_6&:=\left(\begin{smallmatrix}
2 & -1 &    &    &    &\\
-1& 2  & -1 &    &    &\\
  & -1 & 2  & -1 &    &-1\\
  &    & -1 & 2  & -1 &\\
  &    &    & -1 &  2 &\\
  &    & -1 &    &    &2\\
\end{smallmatrix}\right),\\
 E_6^\star &:= 3E_6^{-1},  \\
S_4&:=\left(\begin{smallmatrix} 
1  &    & 3/2 &   \\
   & 1  &     &3/2\\
3/2&    & 3   &   \\
   & 3/2&     & 3 \\
\end{smallmatrix}\right),\text{ and }\\ 
Q(x,y)&:= \left(\begin{smallmatrix}
{}^t xSx & {}^txSy\\
{}^tySx  & {}^tySy\\
\end{smallmatrix}\right).
\end{align*}
The modular forms of note are
\begin{align*}
\alpha_1 &:= \theta_{A_2},\\
\beta_3 &:= \theta_{E_6} - 10\alpha_1^3+ 9 \theta_{E_6^*},\\
\delta_3 &:= \theta_{E_6}-9\theta_{E_6^*},\\
\gamma_4(Z)&:=\sum_{\substack{(x_1,\dots,x_4)\in\Z^4\\(y_1,\dots, y_4)\in \Z^4}}(c^2-d^2) e\left(\tr (Q(x,y)Z
   )\right) \text{ where}\\
   &\phantom{xxxx}c:=(x_1y_3-x_3y_1)+(x_2y_4-x_4y_2) \text{ and}\\
   &\phantom{xxxx}d:=(x_1y_4-x_4y_1)+(x_3y_2-x_2y_3)  \\
   \chi_{14} &:= \{\alpha_1,\beta_3,\delta_3,\gamma_4\}.
\end{align*}

We observe that these forms, respectively, have weights 1, 3, 3, 4, 14.  With this notation, then,
\[
M_*(\Gamma^{\psi_3}_0(3)) =
\C[\alpha_1,\beta_3,\delta_3,\gamma_4]\oplus \chi_{14}\C[\alpha_1,\beta_3,\delta_3,\gamma_4].
\]

\subsubsection{Siegel modular forms of degree 2 and level 4}
Now, we consider $\Gamma^{\psi_4}_0(4)$.  First, as in \cite{Igusa3}, let
\begin{multline*}
\chi_5:=\theta_{(0,0),(0,0)}(Z)\theta_{(0,0),(0,1)}(Z)\theta_{(0,0),(1,0)}(Z)\theta_{(0,0),(1,1)}(Z)\theta_{(0,1),(0,0)}(Z)\\
\times \theta_{(0,1),(1,0)}(Z)\theta_{(1,0),(0,0)}(Z)\theta_{(1,0),(0,1)}(Z)\theta_{(1,1),(0,0)}(Z)\theta_{(1,1),(1,1)}(Z).
\end{multline*}
Then, as in \cite{Ibukiyama}, let
\begin{align*}
f_{1/2}&:=\theta_{(0,0),(0,0)}(2Z)\\
f_1&:=f_{1/2}^2\\
g_2&:=\theta_{(0,0),(0,0)}(2Z)^4+\theta_{(0,1),(0,0)}(2Z)^4+\theta_{(1,0),(0,0)}(2Z)^4+\theta_{(1,1),(0,0)}(2Z)^4\\
h_2&:=\theta_{(0,0),(0,0)}(2Z)^4+\theta_{(0,0),(0,1)}(2Z)^4+\theta_{(0,0),(1,0)}(2Z)^4+\theta_{(0,0),(1,1)}(2Z)^4\\
f_3&:=(\theta_{(0,0),(0,1)}(2Z)\theta_{(0,0),(1,0)}(2Z)\theta_{(0,0),(1,1)}(2Z))^2\\
\chi_{11}&:=\{f_1,g_2,h_2,f_3\}.
\end{align*}
Then $M_*(\Gamma_0^{\psi_4}(4))=\C[f_1,g_2,h_2,f_3,\chi_{11}]$.

\subsection{Particular spaces of vector valued Forms}\label{ssec:spaces}

The ring of all vector valued Siegel modular forms $\bigoplus_{k,j} M_{k,j}(\Gamma)$ is not finitely generated.  For this reason the vector valued weight $j$ is usually fixed.  The resulting module is finitely generated over $\bigoplus_k M_k(\Gamma)$.

Before we summarize the results on such modules of vector valued modular forms, we remark that vector valued Siegel modular forms are always cuspidal.  This follows from a much more general result by Weissauer \cite{Weissauer}, that covers Siegel modular forms of all degrees.

We summarize the spaces that we can compute in Table~\ref{tbl:vvsmfs}.

\begin{table}
\begin{tabular}{ll}\hline\hline
Group and Weights & Reference\\\hline
$\Gamma,\;j = 2,\;k \in \Z$ & \cite{Satoh}\\\hline
$\Gamma,\;j = 4,6,\;k \in \Z$ & \cite{IbukiyamaSym46}\\\hline
\end{tabular}
\caption{Spaces of vector valued Siegel modular forms whose generators
are known.}\label{tbl:vvsmfs}
\end{table}

\subsubsection{Satoh's Theorem}
By Satoh \cite{Satoh} we know that $\bigoplus_k M_{k,2}(\Gamma)$ is generated by $6$ elements, all of which can be expressed in terms of Satoh brackets (see Section \ref{ssec:satoh-bracket}).  More precisely, he shows that
\begin{align*} 
  M_{k,2}(\Gamma)
=&
  [E_4, E_6] \cdot M_{k-10}(\Gamma) \oplus
  [E_4, \chi_{10}] \cdot M_{k-14}(\Gamma) \oplus
\\&
  [E_4, \chi_{12}] \cdot M_{k-16}(\Gamma)] \oplus
  [E_6, \chi_{10}] \cdot \C[E_6, \chi_{10}, \chi_{12}]_{k-16} \oplus
\\&
  [E_6, \chi_{12}] \cdot \C[E_6, \chi_{10}, \chi_{12}]_{k-18} \oplus
  [\chi_{10}, \chi_{12}] \cdot \C[\chi_{10}, \chi_{12}]_{k-22}
\text{.}
\end{align*}
By $\C[A_1,\dots,A_n]_k$ we mean the module of weight $k$ modular
  forms that can be expressed in terms of generators $A_1,\dots,A_n$.

\subsubsection{Ibukiyama's Theorems}

In \cite{IbukiyamaSym46}, generators for the rings $M_{k,4}$ and
$M_{k,6}$ are given.  The ring of forms of weight $(k,4)$ is generated
by 10 forms that are defined in terms of differential operators
similar to the Satoh bracket mentioned above.  The ring of forms of
weight $(k,6)$ are generated by a Klingen Eisenstein series, two theta
series with pluriharmonics (see Section~\ref{ssec:vvthetaseries} for
more information), and four forms that are defined via differential operators.



\section{A particular implementation}\label{sec:implementation}

In this section we describe our implementation of a package that can
handle a wide variety of Siegel modular forms.  In particular, one
can multiply Satoh brackets by scalar-valued Siegel modular forms and one
can multiply products of theta constants by other similar products.
In short, our package allows for the multiplication of any two formal Siegel
modular forms (even if the formal Siegel modular forms do not actually
correspond to actual Siegel modular forms).  Of particular note is
that our implementation handles Siegel modular forms compactly and
efficiently.

\subsection{An implementation in Sage}

A Siegel modular form is
implemented in the Sage package~\cite{smf-package} as a formal Siegel
modular form, i.e.~as a map~$C:Q\rightarrow R$, where $R$ is a module
(or ring) with a $\GL$-action, such that
\begin{equation*}
  A^{-1}.C(A.f)
  =
  C(f)
\end{equation*}
for all $A$ in $\GL$ and all $f$ in $Q$. Such a map can be readily
implemented via a class {\tt SiegelModularForms\_class} encapsulating
a (Python-)dictionary whose keys are the $\GL$-reduced quadratic forms
below a certain bound. The methods of this class include
multiplication, addition, Rankin-Cohen brackets, Hecke operators,
etc. The bottlenecks for an effective implementation of these methods
are multiplication and reduction of integral binary
quadratic forms.  

\begin{remark}
We observe that our implementation of Siegel modular forms in Sage
allows for the precision of a Siegel modular form $F$ to be described
in a number of ways.  When we say $F$ has discriminant precision $X$
(here $X$ is a positive integer), we mean that the keys of the
dictionary mentioned above are the $\GL$-reduced quadratic forms of
discriminant greater than $-X$.

There are a handful of other ways to describe the precision of a
Siegel modular form (e.g., trace and box precisions).  In
\cite{smf-package} we include a Siegel modular form precision class
that handles translations between the various kinds of precision.

For many theoretical purposes, it is best to describe the precision
of a Siegel modular form in terms of dyadic trace
\cite{PoorYuenExtremeCore,PoorYuenDimensions,PoorYuen2,PoorYuen,PoorYuen3}.
Poor and Yuen have exploited this idea to prove, among other results,
a nonvanishing theorem that can be thought of as a generalization of
Sturm's bound for classical modular forms.\end{remark}

\subsection{Multiplication}

The multiplication of Siegel modular forms is expensive because it
involves the evaluation of triple sums, whereas the reduction has to
be applied in almost every call to a method of the class {\tt
  SiegelModularForms\_class} since its instances only store
coefficients for reduced forms.  

The obvious algorithm for the
reduction is a variant of the Euclidean algorithm: applying repeatedly
to a form $f$ the two steps $[a,b,c]\mapsto [a,b\%a,c]$ (where $b\%a$
denote the Euclidean remainder of $b$ modulo $a$) and
$[a,b,c]\mapsto[c,b,a]$ yields eventually the reduced form which is
$\GL$-equivalent to $f$.  

For a fast multiplication it is useful to
provide separate implementations according to the types of the values
of the formal Siegel modular forms. Operations on integer valued
Siegel modular forms can clearly be handled faster than on those
taking values, for example, in a polynomial ring over a number field.

We point out that by ``fast multiplication" we do not mean Karatsuba multiplication or the like.  There are Karatsuba-like algorithms for multivariate power series rings, but they do not take advantage of the invariance of the coefficients for the Siegel modular form.  The (potential) gains won by implementing a Siegel modular form as a multivariate power series rings that has a Karatsuba-like multiplication are outweighed by the great number more coefficients such an object would have to store and the number of redundant multiplications that would have to be done.  There is currently no known Karatsuba-type algorithm that respects the invariance of the coefficients of a Siegel modular form, but this will investigated in future work by the first author.

As a final remark we note that the set of all formal Siegel modular
forms for a given coefficient ring $R$, i.e.~the set of formal Siegel
modular forms $C:Q \rightarrow R(\chi)$ ($\chi$ a linear character of
$\GL$), is naturally equipped with a multiplication. Namely, for two
such formal Siegel modular forms $C_1:Q \rightarrow R(\chi_1)$ and
$C_2:Q \rightarrow R(\chi_2)$ the map $C$ defined by
\begin{equation*}
  C(f)
  =
  \sum_{\substack{f_1,f_2\in Q \\ f = f_1+f_2
    }} C_1(f_1)C_2(f_2)
\end{equation*}
defines a formal Siegel modular form $C:Q\rightarrow
R(\chi_1\chi_2)$.  Note that we would like to allow different
characters since we would like to be able to multiply, for example, a
formal Siegel modular form corresponding to even weight scalar-valued
Siegel modular form (and hence with trivial character) with formal
Siegel modular forms corresponding to scalar-valued Siegel modular
forms of odd weight (and hence with character $\sym{det}$). 

However,
if $\chi_1$ is different from $\chi_2$ there is no natural sum of
$C_1$ and $C_2$. On the other hand in our Sage implementation it is
desirable to view formal Siegel modular forms as elements of an
ambient algebra. This allows for an implementation which is consistent
with Sage's internal coercion system, which can take then over the
necessary coercion steps to multiply (and add) formal Siegel modular
forms with different coefficient rings if the coefficients can be coerced to
a common ring. 

To view formal Siegel modular of a given coefficient
ring as elements of an algebra we proceed as follows. Let $\Xi$ be the
group of linear characters of $\GL$. We view a formal Siegel modular
form with values in $R(\chi)$ as a formal Siegel modular form with
values in the group ring $R[\Xi]$ equipped with the $\GL$-action
\begin{equation*}
  (A,\sum_{\chi\in \Xi} c(\chi)e_\chi)
  \mapsto
  \sum_{\chi\in\Xi}\chi(A)\,A.c( \chi)\,e_\chi
\end{equation*}
via the natural embedding $R(\chi)\rightarrow R[\Xi]$ which maps $r$
to $re_\chi$.  Here $e_\chi$ runs through the natural basis of the
group ring $R[\Xi]$. Formal Siegel modular forms with values in
$R[\Xi]$ then naturally form an algebra over the base ring of $R$
(which is either $R$ if the action of $\GL$ on $R$ is trivial or the
base ring $R'$ of $R$ if $R$ is a polynomial ring in two variables
over $R'$ equipped with the natural $\GL$-action) via the Cauchy
product mentioned above and the obvious addition.

\subsection{Data}

At \cite{smf-tables}, one can find data for scalar and vector valued
Siegel modular forms of level 1.  In the future, similar data for
Siegel modular forms of level greater than 1 will be posted.  In
particular, it is our goal to have data for all the rings and spaces
listed in Tables~\ref{tbl:vvsmfs} and \ref{tbl:subgroups}.


\section{Generalizations}\label{sec:generalizations}
A class similar to the {\tt SiegelModularForms\_class} would work also
for other types of automorphic forms like Hilbert modular forms,
Siegel modular forms of higher degree, orthogonal modular forms,
etc. We formulate here in an abstract way the kind of object should
be implemented to treat all kinds of higher rank automorphic forms in a unified way.

Let $G$ be a group, let $R$ be a module commutative ring and $M$ be
a monoid, both equipped with a $G$-action $(g,r)\mapsto g.r$
respectively $(g,m)\mapsto g.m$ (such that, for each $g$ in $G$ the
maps $r\mapsto g.r$ and $m\mapsto g.m$ are homomorphisms of modules or
rings and monoids, respectively).  We set
\begin{equation*}
  R[\![ M]\!]^G
  :=
  \{C:M\rightarrow R\;:\; \forall g\in G, m\in M:C(g.m)=g.C(m)\}
\end{equation*}
Note that $R[\![ M]\!]^G$ , for a ring $R$, is a subring of the ring
of power series $R[\![ M]\!]$.

In fact, the types of automorphic forms mentioned above and many more
possess a Fourier expansion whose Fourier coefficients lie in
$R[\![M]\!]^G$ for suitable choices of $M$, $R$ and $G$.
Table~\ref{tab:equivariant-power-series} summarizes various examples.
An implementation of these objects would again be a class built around
a dictionary whose keys are representatives for the orbits in
$G\backslash M$. The effectiveness of such an implementation depends,
of course, on a good reduction theory for finding distinguished
representatives for the classes in $G\backslash M$. Such a class would
be highly desirable and encourage the implementation of more types of
automorphic forms. An implementation for Sage by the first author will appear soon.

\begin{table}[h]
    \begin{tabular}{|l|l|p{5cm}|}
      \hline\hline      
      
      \multirow{5}{3cm}{Elliptic modular forms} & Group $G$ & 1\\
      & Monoid $M$  & $\Z_{\geq 0}$\\
      & $G$-action on $M$&\\
      & Module (or ring) $R$ &$\F$\\ 
      & $G$-action on $R$ & \\\hline

      \multirow{5}{3cm}{Vector-valued Siegel modular forms of degree $n$ and weight $k,j$} & Group $G$ &  $\sym{GL}(n,\Z)$\\
      & Monoid $M$ & Set of semi-positive definite integral quadratic forms $f$ in $n$ variables\\
      & $G$-action on $M$ & $(g,f)\mapsto f\big((X_1,\dots,X_n)g\big)$\\
      & Module (or ring) $R$ & $\F[X_1,\dots,X_n]_j$\\
      & $G$-action on $R$ &$(g,p)\mapsto\det(g)^k \times p(\big((X_1,\dots,X_n)g\big)$\\\hline

      \multirow{5}{3cm}{Hilbert modular forms of (parallel) weight $k$ over a totally real number field $L$} & Group $G$ & $\Z_L^*$\\
      & Monoid $M$ & Set of totally positive or zero elements in the inverse different of $L$\\
      & $G$-action on $M$ & $(g,a) \mapsto g^2a$\\
      & Module (or ring) $R$ & $\F$\\
      & $G$-action on $R$ &$(g,r)\mapsto N(g)^kr$\\\hline
          
      \multirow{5}{3cm}{Hermitian modular forms over the imaginary quadratic field~$L$} & Group $G$ & $\sym{GL}(n, \Z_{L})$\\
      & Monoid $M$ & Set of semi-positive definite integal hermitian forms ~$f$ over $L$ with $n$ variables\\
      & $G$-action on $M$ & $(g,f)\mapsto f\big((X_1,\dots,X_n)g\big)$\\
      & Module (or ring) $R$ & $\F$\\
      & $G$-action on $R$ &$(g,p) \mapsto \det (g) ^kp$\\\hline
          
      \multirow{5}{3cm}{Jacobi forms of weight $k$ and index $m\gg0$ (in $\vartheta^{-1}$, where $\vartheta$ denotes the different of $L$) over a number field L (cf.~\cite{Boylan})}& Group $G$ & $\Z_L^*\ltimes 2m\Z_L$\\
      & Monoid $M$ & Set of $(D,r)$ in $\vartheta^{-2}\times\vartheta^{-1}$ s.t. $D\equiv r^2\bmod 4m\vartheta^{-1}$, $D=0$ or $-D\gg0$\\
      & $G$-action on $M$ & $\big((\varepsilon,x),(D,r)\big) \mapsto \big((\varepsilon^2D,\varepsilon (r+x))\big)$\\
      & Module (or ring) $R$ & $\F$\\
      & $G$-action on $R$ &$\big((\varepsilon,x),r\big) \mapsto N(\varepsilon)^kr$\\\hline                    
 \end{tabular}
    \caption{Types of Fourier coefficients for various types of
      automorphic forms ($\F$ denotes a commutative ring)}\label{tab:equivariant-power-series}
\end{table}

Most useful for explicit computations are those graded (modules over)
rings of automorphic forms of the types mentioned in the table where
generators are known and can be computed effectively.  There are known
examples, in addition to the one of Siegel modular forms of degree 2 already considered, of spaces of Hilbert modular forms, Jacobi forms, and Siegel modular forms of higher degree whose
generators are known.  More recent examples include Hermitian modular
forms on the full hermitian modular group
\begin{equation*}
  \Gamma_L
  :=
  \{M \in M (4, \Z_L) \;:\; \overline{M}^t J M = J\}
\end{equation*}
with $J := \left(\begin{smallmatrix}0 & I \\ -I & 0
  \end{smallmatrix}\right)$, and where $\Z_L$ is either the maximal
order in $L=\Q (\sqrt{-1})$ \cite{Aoki}, \cite{Krieg} or the maximal
order in $L=\Q(\sqrt{-3})$~\cite{Krieg}.  The generators can, in these
cases, be obtained as Hermitian Maass lifts or Hermitian theta series.


\bibliography{formal_smfs}
\bibliographystyle{plain}

\affiliationone{        
Martin Raum\\
Max-Planck-Institut f\"ur Mathematik\\
Bonn\\
Germany\\
\email{mraum@mpim-bonn.mpg.de}
}
\affiliationtwo{
Nathan C. Ryan\\
Department of Mathematics\\
Bucknell University\\
USA\\
\email{nathan.ryan@bucknell.edu}\\
}
\affiliationthree{
Nils-Peter Skoruppa\\
Fachbereich Mathematik\\
Universit\"at Siegen\\
Germany\\
\email{nils.skoruppa@gmail.de}\\
}
\affiliationthree{
Gonzalo Tornar\'ia\\
Centro de Matem\'atica\\
Universidad de la Rep\'ublica\\
Uruguay\\
\email{tornaria@cmat.edu.uy}
}

\end{document}